\documentclass[a4paper,12pt]{amsproc}
\usepackage{amssymb}
\usepackage{tikz}
\usepackage[hyphens]{url} 

\theoremstyle{plain}
 \newtheorem{thm}{Theorem}[section]
 
 \newtheorem{cor}{Corollary}[section]
 \newtheorem{lem}{Lemma}[section]
\theoremstyle{definition}
 
 \newtheorem{dfn}{Definition}[section]
\theoremstyle{remark}
 
 \numberwithin{equation}{section}
\renewcommand{\leq}{\leqslant}
\renewcommand{\geq}{\geqslant}

\title{Cohomology Group of $K(\mathbb{Z},4)$ and $K(\mathbb{Z},5)$}
\author{ K.Salehi\\
Massey university}

\begin{document}

\begin{abstract}
We are familiar with properties and structure of topological spaces. One of the powerful tools, which help us to figure out the structure of topological spaces is (Leray- Serre) spectral sequence. Although Eilenberg-Maclane space plays important roles in topology, and respectively geometry.  Actually finding cohomology groups of this space can be useful for classifying spaces, and also homotopy groups structure of these groups. 
This paper discusses  how to compute cohomology groups of Eilenberg-Maclane spaces $K(\mathbb{Z}, 4)$ and $K(\mathbb{Z}, 5)$ (cohomology degree less than $11$).  Furthermore we give the method to find cohomology groups of $K(\mathbb{Z}, n)$.  Some proofs are given to the basic facts about cohomology group of $K(\mathbb{Z}, 5)$ and $K(\mathbb{Z},4).$
\end{abstract} 
\maketitle
\begin{flushleft}
\section{Introduction}
Algebraic topology is one of the important branches of mathematics that examines the structure and properties of topological space. One of the tools that playing important roles is serre spectral sequences. Serre spectral sequences are powerful theoretical and computational tool with numerous applications to algebraic topology. Indeed, a main application is the computation of various cohomology groups of topological spaces.
We know that cohomology groups of $K (\mathbb{Z},1)$, $K (\mathbb{Z},2)$ and  $K (\mathbb{Z},3)$ are computed (see\cite{a}). In this paper first we introduce the notion of a spectral sequence (sections $2$, $3$), next in section $4$, we will introduce some lemma and then prove them by using spectral sequence, to compute the cohomology groups of $K(\mathbb{Z}, 4)$. In section $5$, we calculate  $K(\mathbb{Z},5)$ with coefficient group $\mathbb{Z}$. Finding cohomology groups of this space with coefficient $\mathbb{Z}$ are complicated, because integer group $\mathbb{Z}$ contains torsion subgroups in terms of $n\mathbb{Z}$.
\section{Preliminaries}
Here we recall without the proof some significant definitions and theorems from algebraic topology and geometry that will be used later in this note. (see \cite{a})
\begin{dfn}\label{11}
A fibration is a map $P: E \longrightarrow B$ which having the homotopy lifting property with respect to all path-connected spaces. (see \cite{c})\\ 
\end{dfn}
\begin{dfn}\label{12}
A space X having just one nontrivial homotopy group, is called an Eilenberg- Maclane space $K(G,n)$. (see \cite{a})
\end{dfn}
\begin{dfn}\label{13}
A topological space is called simply- connected if it is path- connected and has trivial fundamental group. (see \cite{b})
\end{dfn}
\begin{dfn}\label{14}
A space with base-point $x_{0}$ is said to be $n$- connected if $$\pi_{i} (X,x_{0}) = 0 \hspace*{.5cm} \forall \hspace*{.1cm} i\leq n.$$
\end{dfn}
\begin{thm}\label{1} (Hurewicz)
If space $X$ is $(n-1)$- connected for $n\geq 2$, then $H _{i} (X)=0$ for $i<n$ and $\pi_{n} (X) \simeq H_{n} (X)$. If a pair $(X,A)$ is $(n-1)$- connected, $n\geq 2$, with A simply- connected and nonempty, then $H_{i} (X,A) \simeq \pi_{i} (X,A) =0$ for $i<n$ and $\pi_{n} (X,A) \simeq H_{n} (X,A)$. (see \cite{c})
\end{thm}
\begin{thm}\label{2} (Universal Coefficients for Cohomology)
If a chain complex $C$ of free abelian groups has homology groups $H_{i} (C)$, then the cohomology groups $H^{i} (C,G)$ of co-chain complex $Hom(C_{i} , G)$ are determined by split exact sequences 
\begin{equation}
0 \longrightarrow Ext(H_{n-1} (X),G) \longrightarrow H^{n} (C,G) \longrightarrow Hom(H_{n} (C) , G) \longrightarrow 0
\end{equation}
In practice, the Ext term either vanishes or is computable. (see \cite{c})
\end{thm}

\begin{thm} (Homology with Coefficient)
If the homology groups $H_{n}$ and $H_{n-1}$ of a chain complex $C$ of free abelian groups are finitely generated, with torsion subgroups $T_{n} \subset H_{n}$ and $T_{n-1} \subset H_{n-1}$ , then $H^{n} (C, \mathbb{Z}) \simeq (\frac{H_{n}}{T_{n}}) \oplus T_{n-1}$\\
In practice, the Tor term either vanishes or is easily computable. (see \cite{a})
\end{thm}
\begin{thm}
If $C$ is a chain complex of free abelian groups, then there are natural short exact sequences 
$$
0 \longrightarrow H_{n} (C) \otimes G \longrightarrow H_{n} (C,G) \longrightarrow Tor(H_{n-1} (C), G) \longrightarrow 0 .$$ (see \cite{b})

\end{thm}
\section{Spectral Sequences} 
A spectral sequence is a tool to compute the cohomology of chain complex. It arises from a filtration of the dual chain complex and it provides an alternative way to determine the cohomology of the dual chain complex. A spectral sequence consists of a sequence of intermediate dual chain complexes called pages $ E_{0},E_{1},E_{2},E_{3},...$, with differentials denoted by $d_{0},d_{1},d_{2},d_{3},...$, such that $E_{r+1} $ is the cohomology of $E_{r}$. The various pages have accessible cohomology groups which form a finer and finer approximation of the cohomology $H$ we wish to find out. This limit process is convergence, in which case the limit page is denoted by $E_{\infty}$. Even if there is convergence to $E_{\infty}$, reconstruction is still needed to obtain $ H$ from $E_{\infty}$. Although the differentials $d_{r}$'s cannot always be all computed, the existence of the spectral sequence often reveals deep facts about the dual chain complex. The spectral sequence and its internal mechanisms can still lead to very useful and deep applications. (see \cite{d})
\begin{thm} (Spectral Sequences for Cohomology)
for a fibration $$F \longrightarrow X \longrightarrow B $$
with $B$ path-connected and $\pi_{1} (B)$ acting trivially on $H^{*} (F,G)$, there is a spectral sequences $\lbrace E^{p,q} _{r} , d_{r} \rbrace$ with:
\begin{enumerate}
\item
$d_{r} : E^{p,q} _{r} \longrightarrow E^{p+r,q-r+1} _{r}  \hspace*{.1cm}and\hspace*{.1cm}  E^{p,q} _{r+1} = \frac{Ker  d_{r}}{Img  d_{r}} \hspace*{.1cm}   at \hspace*{.1cm}   E^{p,q} _{r}.$\\
\item
Stable terms $E^{p,n-p} _{\infty}$ is isomorphic to the successive quotients $ \frac{F^{n} _{p}}{F^{n} _{p+1}}$ in a filtration\\
$0 \subset F^{n} _{n} \subset F^{n} _{n-1} \subset... \subset F^{n} _{0} = H^{n} (X,G)  \hspace*{.1cm} of \hspace*{.1cm} H^{n} (X,G)$\\
\item 
$E^{p,q} _{2} \approx H^{p} (B,H^{q} (F,G)).$ (see\cite{c})
\end{enumerate}
\end{thm}
\section{Cohomology  Groups of $K(\mathbb{Z},4)$ Via Fibration}
In algebraic topology, for any path connected space $(X,x_{0})$, there is a path fibration
$$\Omega X \longrightarrow PX \longrightarrow X$$
where $X$ is the base space, $PX$ is the total contractible space and $\Omega X$ is fiber over the base space $(X,x_{0})$, which called loop space.
Now consider base space $X=K(\mathbb{Z},4)$, so there is a fibration 
$$K(\mathbb{Z} ,3) = \Omega K(\mathbb{Z},4) \longrightarrow  PK(\mathbb{Z},4) \longrightarrow K(\mathbb{Z},4).$$
Since we know that the topological space $K(\mathbb{Z} ,4)$ is ($3$ - connected), so by definition (\ref{14}) and theorem (\ref{1}),
$$H_{1}(K(\mathbb{Z},4),\mathbb{Z}) \simeq H_{2}(K(\mathbb{Z},4),\mathbb{Z}) \simeq H_{3}(K(\mathbb{Z},4),\mathbb{Z}) \simeq 0 .$$
Now, with the use of theorem (\ref{2}), it is easy to show that
$$H^{1}(K(\mathbb{Z},4),\mathbb{Z}) \simeq H^{2}(K(\mathbb{Z},4),\mathbb{Z}) \simeq H^{3}(K(\mathbb{Z},4),\mathbb{Z}) \simeq 0.$$
On the other hand by Hurewicz theorem we see that 
$$\pi_{4}(K(\mathbb{Z},4)) \simeq H_{4}(K(\mathbb{Z},4),\mathbb{Z}) \simeq \mathbb{Z}.$$
Then from universal coefficient theorem we have
\begin{equation*}
H^{4}(K(\mathbb{Z},4),\mathbb{Z}) \simeq Ext(H_{3} (K(\mathbb{Z},4),\mathbb{Z})) \oplus Hom(H_{4}(K(\mathbb{Z},4),\mathbb{Z})) 
\simeq \mathbb{Z}.
\end{equation*}
By setting  
$$E^{p,q} _{2} := H^{p}(K(\mathbb{Z},4);H^{q} (K(\mathbb{Z},3))$$
and using theorems (\ref{1}), (\ref{2}) and cohomology groups of topological space $K(\mathbb{Z},3)$ (see\cite{a}), we obtain the following results:
$$E^{0,0} _{2} =H^{0} (K(\mathbb{Z},4) ,H^{0} (K(\mathbb{Z},3),\mathbb{Z}))= H^{0} (K(\mathbb{Z},4) ,\mathbb{Z}) \simeq \mathbb{Z}$$
$$E^{0,1} _{2} =H^{0} (K(\mathbb{Z},4) ,H^{1} (K(\mathbb{Z},3),\mathbb{Z}))= H^{0} (K(\mathbb{Z},4) ,0) \simeq 0$$
$$E^{0,2} _{2} =H^{0} (K(\mathbb{Z},4) ,H^{2} (K(\mathbb{Z},3),\mathbb{Z}))= H^{0} (K(\mathbb{Z},4) ,0) \simeq 0$$
$$E^{0,3} _{2} =H^{0} (K(\mathbb{Z},4) ,H^{3} (K(\mathbb{Z},3),\mathbb{Z}))= H^{0} (K(\mathbb{Z},4) ,\mathbb{Z}) \simeq \mathbb{Z}$$
$$E^{0,4} _{2} =H^{0} (K(\mathbb{Z},4) ,H^{4} (K(\mathbb{Z},3),\mathbb{Z}))= H^{0} (K(\mathbb{Z},4) ,0) \simeq 0$$
$$E^{0,5} _{2} =H^{0} (K(\mathbb{Z},4) ,H^{5} (K(\mathbb{Z},3),\mathbb{Z}))= H^{0} (K(\mathbb{Z},4) ,0) \simeq 0$$
$$E^{0,6} _{2} =H^{0} (K(\mathbb{Z},4) ,H^{6} (K(\mathbb{Z},3),\mathbb{Z}))= H^{0} (K(\mathbb{Z},4) ,\mathbb{Z}_{2}) \simeq \mathbb{Z}_{2}$$
$$E^{0,7} _{2} =H^{0} (K(\mathbb{Z},4) ,H^{7} (K(\mathbb{Z},3),\mathbb{Z}))= H^{0} (K(\mathbb{Z},4) ,0) \simeq 0$$
$$E^{0,8} _{2} =H^{0} (K(\mathbb{Z},4) ,H^{8} (K(\mathbb{Z},3),\mathbb{Z}))= H^{0} (K(\mathbb{Z},4) ,\mathbb{Z}_{3}) \simeq \mathbb{Z}_{3}$$
$$E^{0,9} _{2} =H^{0} (K(\mathbb{Z},4) ,H^{9} (K(\mathbb{Z},3),\mathbb{Z}))= H^{0} (K(\mathbb{Z},4) ,\mathbb{Z}_{2}) \simeq \mathbb{Z}_{2}$$\\
$$E^{1,0} _{2} =H^{1} (K(\mathbb{Z},4) ,H^{0} (K(\mathbb{Z},3),\mathbb{Z}))= H^{1} (K(\mathbb{Z},4) ,\mathbb{Z}) \simeq 0$$
$$E^{1,1} _{2} =H^{1} (K(\mathbb{Z},4) ,H^{1} (K(\mathbb{Z},3),\mathbb{Z}))= H^{1} (K(\mathbb{Z},4) ,0) \simeq 0$$
$$E^{1,2} _{2} =H^{1} (K(\mathbb{Z},4) ,H^{2} (K(\mathbb{Z},3),\mathbb{Z}))= H^{1} (K(\mathbb{Z},4) ,0) \simeq 0$$
$$E^{1,3} _{2} =H^{1} (K(\mathbb{Z},4) ,H^{3} (K(\mathbb{Z},3),\mathbb{Z}))= H^{1} (K(\mathbb{Z},4) ,\mathbb{Z}) \simeq 0$$
$$E^{1,4} _{2} =H^{1} (K(\mathbb{Z},4) ,H^{4} (K(\mathbb{Z},3),\mathbb{Z}))= H^{1} (K(\mathbb{Z},4) ,0) \simeq 0$$
$$E^{1,5} _{2} =H^{1} (K(\mathbb{Z},4) ,H^{5} (K(\mathbb{Z},3),\mathbb{Z}))= H^{1} (K(\mathbb{Z},4) ,0) \simeq 0$$
$$E^{1,6} _{2} =H^{1} (K(\mathbb{Z},4) ,H^{6} (K(\mathbb{Z},3),\mathbb{Z}))= H^{1} (K(\mathbb{Z},4) ,\mathbb{Z}_{2}) \simeq 0$$
$$E^{1,7} _{2} =H^{1} (K(\mathbb{Z},4) ,H^{7} (K(\mathbb{Z},3),\mathbb{Z}))= H^{1} (K(\mathbb{Z},4) ,0) \simeq 0$$
$$E^{1,8} _{2} =H^{1} (K(\mathbb{Z},4) ,H^{8} (K(\mathbb{Z},3),\mathbb{Z}))= H^{1} (K(\mathbb{Z},4) ,\mathbb{Z}_{3}) \simeq 0$$
$$E^{1,9} _{2} =H^{1} (K(\mathbb{Z},4) ,H^{9} (K(\mathbb{Z},3),\mathbb{Z}))= H^{1} (K(\mathbb{Z},4) ,\mathbb{Z}_{2}) \simeq 0$$\\
$$E^{2,0} _{2} =H^{2} (K(\mathbb{Z},4) ,H^{0} (K(\mathbb{Z},3),\mathbb{Z}))= H^{2} (K(\mathbb{Z},4) ,\mathbb{Z}) \simeq 0$$
$$E^{2,1} _{2} =H^{2} (K(\mathbb{Z},4) ,H^{1} (K(\mathbb{Z},3),\mathbb{Z}))= H^{2} (K(\mathbb{Z},4) ,0) \simeq 0$$
$$E^{2,2} _{2} =H^{2} (K(\mathbb{Z},4) ,H^{2} (K(\mathbb{Z},3),\mathbb{Z}))= H^{2} (K(\mathbb{Z},4) ,0) \simeq 0$$
$$E^{2,3} _{2} =H^{2} (K(\mathbb{Z},4) ,H^{3} (K(\mathbb{Z},3),\mathbb{Z}))= H^{2} (K(\mathbb{Z},4) ,\mathbb{Z}) \simeq 0$$
$$E^{2,4} _{2} =H^{2} (K(\mathbb{Z},4) ,H^{4} (K(\mathbb{Z},3),\mathbb{Z}))= H^{2} (K(\mathbb{Z},4) ,0) \simeq 0$$
$$E^{2,5} _{2} =H^{2} (K(\mathbb{Z},4) ,H^{5} (K(\mathbb{Z},3),\mathbb{Z}))= H^{2} (K(\mathbb{Z},4) ,0) \simeq 0$$
$$E^{2,6} _{2} =H^{2} (K(\mathbb{Z},4) ,H^{6} (K(\mathbb{Z},3),\mathbb{Z}))= H^{2} (K(\mathbb{Z},4) ,\mathbb{Z}_{2}) \simeq 0$$
$$E^{2,7} _{2} =H^{2} (K(\mathbb{Z},4) ,H^{7} (K(\mathbb{Z},3),\mathbb{Z}))= H^{2} (K(\mathbb{Z},4) ,0) \simeq 0$$
$$E^{2,8} _{2} =H^{2} (K(\mathbb{Z},4) ,H^{8} (K(\mathbb{Z},3),\mathbb{Z}))= H^{2} (K(\mathbb{Z},4) ,\mathbb{Z}_{3}) \simeq 0$$
$$E^{2,9} _{2} =H^{2} (K(\mathbb{Z},4) ,H^{9} (K(\mathbb{Z},3),\mathbb{Z}))= H^{2} (K(\mathbb{Z},4) ,\mathbb{Z}_{2}) \simeq 0$$\\
$$E^{3,0} _{2} =H^{3} (K(\mathbb{Z},4) ,H^{0} (K(\mathbb{Z},3),\mathbb{Z}))= H^{3} (K(\mathbb{Z},4) ,\mathbb{Z}) \simeq 0$$
$$E^{3,1} _{2} =H^{3} (K(\mathbb{Z},4) ,H^{1} (K(\mathbb{Z},3),\mathbb{Z}))= H^{3} (K(\mathbb{Z},4) ,0) \simeq 0$$
$$E^{3,2} _{2} =H^{3} (K(\mathbb{Z},4) ,H^{2} (K(\mathbb{Z},3),\mathbb{Z}))= H^{3} (K(\mathbb{Z},4) ,0) \simeq 0$$
$$E^{3,3} _{2} =H^{3} (K(\mathbb{Z},4) ,H^{3} (K(\mathbb{Z},3),\mathbb{Z}))= H^{3} (K(\mathbb{Z},4) ,\mathbb{Z}) \simeq 0$$
$$E^{3,4} _{2} =H^{3} (K(\mathbb{Z},4) ,H^{4} (K(\mathbb{Z},3),\mathbb{Z}))= H^{3} (K(\mathbb{Z},4) ,0) \simeq 0$$
With the above results we can compute page 2 as follows:
\begin{center}
\begin{figure}
\begin{tikzpicture}[scale=.15]
\draw[->] (0,-60)--(60,-60);
\draw[->] (0,-60)--(0,0);
\node at (2,-57){$\mathbb{Z}$};
\node at (2,-52){$0$};
\node at (2,-47){$0$};
\node at (2,-42){$\mathbb{Z}$};
\node at (2,-37){$0$};
\node at (2,-32){$0$};
\node at (2,-27){$\mathbb{Z}_{2}$};
\node at (2,-22){$0$};
\node at (2,-17){$\mathbb{Z}_{3}$};
\node at (2,-12){$\mathbb{Z}_{2}$};
\node at (2,-7){$\mathbb{Z}_{2}$};
\node at (7,-57){$0$};
\node at (7,-52){$0$};
\node at (7,-47){$0$};
\node at (7,-42){$0$};
\node at (7,-37){$0$};
\node at (7,-32){$0$};
\node at (7,-27){$0$};
\node at (7,-22){$0$};
\node at (7,-17){$0$};
\node at (7,-12){$0$};
\node at (7,-7){$0$};
\node at (12,-57){$0$};
\node at (12,-52){$0$};
\node at (12,-47){$0$};
\node at (12,-42){$0$};
\node at (12,-37){$0$};
\node at (12,-32){$0$};
\node at (12,-27){$0$};
\node at (12,-22){$0$};
\node at (12,-17){$0$};
\node at (12,-12){$0$};
\node at (12,-7){$0$};
\node at (17,-57){$0$};
\node at (17,-52){$0$};
\node at (17,-47){$0$};
\node at (17,-42){$0$};
\node at (17,-37){$0$};
\node at (17,-32){$0$};
\node at (17,-27){$0$};
\node at (17,-22){$0$};
\node at (17,-17){$0$};
\node at (17,-12){$0$};
\node at (17,-7){$0$};;
\node at (22,-57){$E^{4,0} _{2}$};
\node at (22,-52){$0$};
\node at (22,-47){$0$};
\node at (22,-42){$E^{4,3} _{2}$};
\node at (22,-37){$0$};
\node at (22,-32){$0$};
\node at (22,-27){$E^{4,6} _{2}$};
\node at (22,-22){$0$};
\node at (22,-17){$E^{4,8} _{2} $};
\node at (22,-12){$E^{4,9} _{2}$};
\node at (22,-7){$...$};
\node at (27,-57){$E^{5,0} _{2}$};
\node at (27,-52){$0$};
\node at (27,-47){$0$};
\node at (27,-42){$E^{5,0} _{2}$};
\node at (27,-37){$0$};
\node at (27,-32){$0$};
\node at (27,-27){$0$};
\node at (27,-22){$0$};
\node at (27,-17){$0$};
\node at (27,-12){$0$};
\node at (27,-7){$...$};
\node at (32,-57){$E^{6,0} _{2}$};
\node at (32,-52){$0$};
\node at (32,-47){$0$};
\node at (32,-42){$E^{6,0} _{2}$};
\node at (32,-37){$0$};
\node at (32,-32){$0$};
\node at (32,-27){$0$};
\node at (32,-22){$0$};
\node at (32,-17){$0$};
\node at (32,-12){$0$};
\node at (32,-7){$...$};
\node at (38,-57){$E^{7,0} _{2}$};
\node at (38,-52){$0$};
\node at (38,-47){$0$};
\node at (38,-42){$E^{7,3} _{2}$};
\node at (38,-37){$0$};
\node at (38,-32){$0$};
\node at (38,-27){$E^{7,6} _{2}$};
\node at (38,-22){$0$};
\node at (38,-17){$E^{7,8} _{2}$};
\node at (38,-12){$E^{7,9} _{2}$};
\node at (38,-7){$...$};
\node at (43,-57){$E^{8,0} _{2}$};
\node at (43,-52){$0$};
\node at (43,-47){$0$};
\node at (43,-42){$E^{8,3} _{2}$};
\node at (43,-37){$0$};
\node at (43,-32){$0$};
\node at (43,-27){$E^{8,6} _{2}$};
\node at (43,-22){$0$};
\node at (43,-17){$E^{8,8} _{2}$};
\node at (43,-12){$E^{8,9} _{2}$};
\node at (43,-7){$...$};
\node at (48,-57){$E^{9,0} _{2}$};
\node at (48,-52){$0$};
\node at (48,-47){$0$};
\node at (48,-42){$E^{9,3} _{2}$};
\node at (48,-37){$0$};
\node at (48,-32){$0$};
\node at (48,-27){$E^{9,6} _{2}$};
\node at (48,-22){$0$};
\node at (48,-17){$E^{9,8} _{2}$};
\node at (48,-12){$E^{9,9} _{2}$};
\node at (48,-7){$...$};
\node at (53,-57){$E^{10,0} _{2}$};
\node at (53,-52){$0$};
\node at (53,-47){$0$};
\node at (53,-42){$E^{10,3} _{2}$};
\node at (53,-37){$0$};
\node at (53,-32){$0$};
\node at (53,-27){$E^{10,6} _{2}$};
\node at (53,-22){$0$};
\draw [->] (2,-28)--(20,-40);
\draw [->] (23,-42)--(41,-56);
\draw [->] (2,-28)--(35,-56);
\draw [->] (2,-19)--(46,-56);
\draw [->] (3,-42)--(20,-56);
\end{tikzpicture}
\caption{page $2$}
\end{figure}
\end{center}
Since the topological space $PK(\mathbb{Z} ,4)$ is contractible, then for any $p$ and $q$, $E^{p,q} _{\infty}$ converges to $H^{p+q} (PK(\mathbb{Z},4),\mathbb{Z})$. In the other words, we have
$$E^{p,q} _{\infty} \Rightarrow H^{p+q} (PK(\mathbb{Z},4),\mathbb{Z}) = 0$$
\begin{lem}
$E^{5,0} _{2} = H^{5} (K(\mathbb{Z},4),\mathbb{Z}) =0$
\end{lem}
\begin{proof}
Since each differential map in page 2 and other pages come and go out $E^{5,0} _{2}$ are zero, then $E^{5,0} _{2}$ will remain to page infinity and converges to $E^{5,0} _{\infty} =0$, then 
$$E^{5,0} _{2} = H^{5} (K(\mathbb{Z},4),\mathbb{Z}) =0$$
\end{proof}
\begin{lem}
$E^{6,0} _{2} = H^{6} (K(\mathbb{Z},4),\mathbb{Z})=0$
\end{lem}
\begin{proof}
Similarly by using the same method, we obtain $E^{6,0} _{2} = H^{6} (K(\mathbb{Z},4),\mathbb{Z})=0$.
\end{proof}
\begin{lem}
$E^{7,0} _{2} = 0$
\end{lem}
\begin{proof}
By the same method which was used for case $E^{5,0} _{2}$.
\end{proof}
\begin{lem} 
$E^{8,0} _{2} = H^{8} (K(\mathbb{Z},4),\mathbb{Z}) \simeq \mathbb{Z}_2$
\end{lem}
\begin{proof}                                      
A spectral sequence $ \lbrace E_{r} ^{p,q} , d_{r} \rbrace $ is a collection of vector space or groups $E_{r} ^{p,q}$, equipped with differential map $d_{r} (i.e,  d_{r} ^{2} =0 )$
$$E_{r+1} ^{p,q} = H (E_{r} ^{p,q} , d_{r})$$
Now consider following  sequence, such that $d_{2} ^{2} =0 $ 
\begin{equation}
0 \longrightarrow \mathbb{Z} _{2} \simeq E^{0,6} _{2} \longrightarrow E^{4,3} _{2} \simeq \mathbb{Z} \longrightarrow E^{8,0} _{2} \longrightarrow 0
\end{equation}
The left differential map in the above exact sequence is injective and multiplicable by $2$, but because of exactness and converging properties, the right map should be surjective, so we have
$$E^{8,0} _{2} = H^{8} (K(\mathbb{Z},4),\mathbb{Z}) \simeq \mathbb{Z}_2$$
\end{proof}
\begin{lem}
$E^{9,0} _{2} = H^{9} (K(\mathbb{Z},4),\mathbb{Z}) = \mathbb{Z}_{3} $
\end{lem}
\begin{proof}
Since $E^{9,0} _{2} = H^{9} (K(\mathbb{Z},4),\mathbb{Z})$ converg to $E^{9,0} _{\infty} =0$, must die too. Similarly the only differential map with a chance to kill it is
$$\mathbb{Z}_{3} \simeq E^{0,8} _{2} \longrightarrow H^{9,0} (K(\mathbb{Z},4),\mathbb{Z})$$
and so it must be an isomorphism.\\
\end{proof}
\begin{lem}\label{89}
$E^{10,0} _{2} \simeq 0 $
\end{lem}
\begin{proof}
All differential maps which come and go out are zero, therefore $E^{10,0} _{2}=E^{10,0} _{\infty}$ converges to zero.
\end{proof}
\begin{lem}
$E^{11,0} _{2} \simeq 0 $
\end{lem}
\begin{proof}
As same as lemma (\ref{89})  .
\end{proof}
\begin{lem}
$E^{12,0} _{2} \simeq  \mathbb{Z}_{2}$
\end{lem}
\begin{proof}
From page $3$, we have a long exact sequence as follow
$$0 \longrightarrow \mathbb{Z} _{2}  \longrightarrow   \mathbb{Z}_2    \longrightarrow   \mathbb{Z}_2  \longrightarrow E^{12,0} _{2} \longrightarrow 0$$
Which the left hand differential map is multiplicable by $3$, so respectively injective and the middle map is multiplicable by $2$, then the right hand map should be an isomorphism.
\end{proof}
 At last we have following results:\\
\begin{cor}
$$E^{1,0} _{2} =H^{1} (K(\mathbb{Z},4) ,H^{0} (K(\mathbb{Z},3),\mathbb{Z}))= H^{1} (K(\mathbb{Z},4) ,\mathbb{Z}) \simeq 0$$
$$E^{2,0} _{2} =H^{2} (K(\mathbb{Z},4) ,H^{0} (K(\mathbb{Z},3),\mathbb{Z}))= H^{2} (K(\mathbb{Z},4) ,\mathbb{Z}) \simeq 0$$
$$E^{3,0} _{2} =H^{3} (K(\mathbb{Z},4) ,H^{0} (K(\mathbb{Z},3),\mathbb{Z}))= H^{3} (K(\mathbb{Z},4) ,\mathbb{Z}) \simeq 0$$
$$E^{4,0} _{2} =H^{4} (K(\mathbb{Z},4) ,H^{0} (K(\mathbb{Z},3),\mathbb{Z}))= H^{4} (K(\mathbb{Z},4) ,\mathbb{Z}) \simeq \mathbb{Z}$$
$$E^{5,0} _{2} =H^{5} (K(\mathbb{Z},4) ,H^{0} (K(\mathbb{Z},3),\mathbb{Z}))= H^{5} (K(\mathbb{Z},4) ,\mathbb{Z}) \simeq 0$$
$$E^{6,0} _{2} =H^{6} (K(\mathbb{Z},4) ,H^{0} (K(\mathbb{Z},3),\mathbb{Z}))= H^{6} (K(\mathbb{Z},4) ,\mathbb{Z}) \simeq 0$$
$$E^{7,0} _{2} =H^{7} (K(\mathbb{Z},4) ,H^{0} (K(\mathbb{Z},3),\mathbb{Z}))= H^{7} (K(\mathbb{Z},4) ,\mathbb{Z}) \simeq 0$$
$$E^{8,0} _{2} =H^{8} (K(\mathbb{Z},4) ,H^{0} (K(\mathbb{Z},3),\mathbb{Z}))= H^{8} (K(\mathbb{Z},4) ,\mathbb{Z}) \simeq \mathbb{Z}_{2}$$
$$E^{9,0} _{2} =H^{9} (K(\mathbb{Z},4) ,H^{0} (K(\mathbb{Z},3),\mathbb{Z}))= H^{9} (K(\mathbb{Z},4) ,\mathbb{Z}) \simeq \mathbb{Z}_{3}$$
$$E^{10,0} _{2} =H^{10} (K(\mathbb{Z},4) ,H^{0} (K(\mathbb{Z},3),\mathbb{Z}))= H^{10} (K(\mathbb{Z},4) ,\mathbb{Z}) \simeq 0$$
$$E^{11,0} _{2} =H^{11} (K(\mathbb{Z},4) ,H^{0} (K(\mathbb{Z},3),\mathbb{Z}))= H^{11} (K(\mathbb{Z},4) ,\mathbb{Z}) \simeq 0$$
$$E^{12,0} _{2} =H^{12} (K(\mathbb{Z},4) ,H^{0} (K(\mathbb{Z},3),\mathbb{Z}))= H^{12} (K(\mathbb{Z},4) ,\mathbb{Z}) \simeq \mathbb{Z}_2$$
\end{cor}
\section{Cohomology  Groups of $K(\mathbb{Z},5)$ Via Fibration}
Now consider base space $X=K(\mathbb{Z},5)$, then 
$$K(\mathbb{Z} ,4) = \Omega K(\mathbb{Z},5) \longrightarrow  PK(\mathbb{Z},5) \longrightarrow K(\mathbb{Z},5)$$

Here $K(\mathbb{Z} ,5)$ is ($4$ - connected), and as we showed above, we obtain the following results
$$H_{1}(K(\mathbb{Z},5),\mathbb{Z}) \simeq H_{2}(K(\mathbb{Z},5),\mathbb{Z}) \simeq H_{3}(K(\mathbb{Z},5),\mathbb{Z}) \simeq H_{4}(K(\mathbb{Z},5),\mathbb{Z}) \simeq 0 $$
$$H^{1}(K(\mathbb{Z},5),\mathbb{Z}) \simeq H^{2}(K(\mathbb{Z},5),\mathbb{Z}) \simeq H^{3}(K(\mathbb{Z},5),\mathbb{Z})\simeq H^{4}(K(\mathbb{Z},5),\mathbb{Z}) \simeq 0$$
$$\pi_{5}(K(\mathbb{Z},5)) \simeq H_{5}(K(\mathbb{Z},5),\mathbb{Z}) \simeq \mathbb{Z}$$
and
$$
H^{5}(K(\mathbb{Z},5),\mathbb{Z}) \simeq Ext(H_{4} (K(\mathbb{Z},5),\mathbb{Z})) \oplus Hom(H_{5}(K(\mathbb{Z},5),\mathbb{Z})) 
\simeq \mathbb{Z}.
$$
Now it is time to calculate higher cohomology groups and respectively homology groups of $ K(\mathbb{Z},5)$. By setting  
$$E^{p,q} _{2} := H^{p}(K(\mathbb{Z},5);H^{q} (K(\mathbb{Z},4))$$
there are following results which are proved.
$$E^{0,0} _{2} =H^{0} (K(\mathbb{Z},5) ,H^{0} (K(\mathbb{Z},4),\mathbb{Z}))= H^{0} (K(\mathbb{Z},5) ,\mathbb{Z}) \simeq \mathbb{Z}$$
$$E^{0,1} _{2} =H^{0} (K(\mathbb{Z},5) ,H^{1} (K(\mathbb{Z},4),\mathbb{Z}))= H^{0} (K(\mathbb{Z},5) ,0) \simeq 0$$
$$E^{0,2} _{2} =H^{0} (K(\mathbb{Z},5) ,H^{2} (K(\mathbb{Z},4),\mathbb{Z}))= H^{0} (K(\mathbb{Z},5) ,0) \simeq 0$$
$$E^{0,3} _{2} =H^{0} (K(\mathbb{Z},5) ,H^{3} (K(\mathbb{Z},4),\mathbb{Z}))= H^{0} (K(\mathbb{Z},5) ,0) \simeq 0$$
$$E^{0,4} _{2} =H^{0} (K(\mathbb{Z},5) ,H^{4} (K(\mathbb{Z},4),\mathbb{Z}))= H^{0} (K(\mathbb{Z},5) , \mathbb{Z}) \simeq \mathbb{Z}$$
$$E^{0,5} _{2} =H^{0} (K(\mathbb{Z},5) ,H^{5} (K(\mathbb{Z},4),\mathbb{Z}))= H^{0} (K(\mathbb{Z},5) ,0) \simeq 0$$
$$E^{0,6} _{2} =H^{0} (K(\mathbb{Z},5) ,H^{6} (K(\mathbb{Z},4),\mathbb{Z}))= H^{0} (K(\mathbb{Z},5) ,0) \simeq 0$$
$$E^{0,7} _{2} =H^{0} (K(\mathbb{Z},5) ,H^{7} (K(\mathbb{Z},4),\mathbb{Z}))= H^{0} (K(\mathbb{Z},5) , \mathbb{Z}_{2}) \simeq \mathbb{Z}_{2}$$
$$E^{0,8} _{2} =H^{0} (K(\mathbb{Z},5) ,H^{8} (K(\mathbb{Z},4),\mathbb{Z}))= H^{0} (K(\mathbb{Z},5) ,\mathbb{Z}) \simeq \mathbb{Z}$$
$$E^{0,9} _{2} =H^{0} (K(\mathbb{Z},5) ,H^{9} (K(\mathbb{Z},4),\mathbb{Z}))= H^{0} (K(\mathbb{Z},5) ,\mathbb{Z}_{3}) \simeq \mathbb{Z}_{3}$$\\
and 
$$E^{5,0} _{2} =H^{5} (K(\mathbb{Z},5) ,H^{0} (K(\mathbb{Z},4),\mathbb{Z}))= H^{5} (K(\mathbb{Z},5) ,\mathbb{Z}) \simeq \mathbb{Z}$$
$$E^{5,1} _{2} =H^{5} (K(\mathbb{Z},5) ,H^{1} (K(\mathbb{Z},4),\mathbb{Z}))= H^{5} (K(\mathbb{Z},5) ,0) \simeq 0$$
$$E^{5,2} _{2} =H^{5} (K(\mathbb{Z},5) ,H^{2} (K(\mathbb{Z},4),\mathbb{Z}))= H^{5} (K(\mathbb{Z},5) ,0) \simeq 0$$
$$E^{5,3} _{2} =H^{5} (K(\mathbb{Z},5) ,H^{3} (K(\mathbb{Z},4),\mathbb{Z}))= H^{5} (K(\mathbb{Z},5) ,0) \simeq 0$$
$$E^{5,4} _{2} =H^{5} (K(\mathbb{Z},5) ,H^{4} (K(\mathbb{Z},4),\mathbb{Z}))= H^{5} (K(\mathbb{Z},5) , \mathbb{Z}) \simeq \mathbb{Z}$$
$$E^{5,5} _{2} =H^{5} (K(\mathbb{Z},5) ,H^{5} (K(\mathbb{Z},4),\mathbb{Z}))= H^{5} (K(\mathbb{Z},5) ,0) \simeq 0$$
$$E^{5,6} _{2} =H^{5} (K(\mathbb{Z},5) ,H^{6} (K(\mathbb{Z},4),\mathbb{Z}))= H^{5} (K(\mathbb{Z},5) ,0) \simeq 0$$
$$E^{5,7} _{2} =H^{5} (K(\mathbb{Z},5) ,H^{7} (K(\mathbb{Z},4),\mathbb{Z}))= H^{5} (K(\mathbb{Z},5) , \mathbb{Z}_{2}) \simeq 0$$
$$E^{5,8} _{2} =H^{5} (K(\mathbb{Z},5) ,H^{8} (K(\mathbb{Z},4),\mathbb{Z}))= H^{5} (K(\mathbb{Z},5) ,\mathbb{Z}) \simeq \mathbb{Z}_2$$
$$E^{5,9} _{2} =H^{5} (K(\mathbb{Z},5) ,H^{9} (K(\mathbb{Z},4),\mathbb{Z}))= H^{5} (K(\mathbb{Z},4) ,\mathbb{Z}_{3}) \simeq \mathbb{Z}_{3}$$
$$E^{5,10} _{2} =H^{5} (K(\mathbb{Z},5) ,H^{10} (K(\mathbb{Z},4),\mathbb{Z}))= H^{5} (K(\mathbb{Z},5) ,0) \simeq 0$$
$$E^{5,11} _{2} =H^{5} (K(\mathbb{Z},5) ,H^{11} (K(\mathbb{Z},4),\mathbb{Z}))= H^{5} (K(\mathbb{Z},5) ,0) \simeq 0$$\\
\begin{lem}
$$E_{2}^{1, q} = E_{2}^{2, q}=E_{2}^{3, q}=E_{2}^{4, q}=0$$
\end{lem}
\begin{proof}
By straight calculation the proof is obvious.
\end{proof}
By considering page $2$ as follow,
\begin{center}
\begin{figure}
\begin{tikzpicture}[scale=.15]
\draw[->] (0,-60)--(60,-60);
\draw[->] (0,-60)--(0,0);
\node at (2,-57){$\mathbb{Z}$};
\node at (2,-52){$0$};
\node at (2,-47){$0$};
\node at (2,-42){$0$};
\node at (2,-37){$\mathbb{Z}$};
\node at (2,-32){$0$};
\node at (2,-27){$0$};
\node at (2,-22){$0$};
\node at (2,-17){$\mathbb{Z}_{2}$};
\node at (2,-12){$\mathbb{Z}_{3}$};
\node at (7,-57){$0$};
\node at (7,-52){$0$};
\node at (7,-47){$0$};
\node at (7,-42){$0$};
\node at (7,-37){$0$};
\node at (7,-32){$0$};
\node at (7,-27){$0$};
\node at (7,-22){$0$};
\node at (7,-17){$0$};
\node at (7,-12){$0$};
\node at (7,-7){$0$};
\node at (12,-57){$0$};
\node at (12,-52){$0$};
\node at (12,-47){$0$};
\node at (12,-42){$0$};
\node at (12,-37){$0$};
\node at (12,-32){$0$};
\node at (12,-27){$0$};
\node at (12,-22){$0$};
\node at (12,-17){$0$};
\node at (12,-12){$0$};
\node at (12,-7){$0$};
\node at (17,-57){$0$};
\node at (17,-52){$0$};
\node at (17,-47){$0$};
\node at (17,-42){$0$};
\node at (17,-37){$0$};
\node at (17,-32){$0$};
\node at (17,-27){$0$};
\node at (17,-22){$0$};
\node at (17,-17){$0$};
\node at (17,-12){$0$};
\node at (17,-7){$0$};;
\node at (22,-57){$0$};
\node at (22,-52){$0$};
\node at (22,-47){$0$};
\node at (22,-42){$0$};
\node at (22,-37){$0$};
\node at (22,-32){$0$};
\node at (22,-27){$0$};
\node at (22,-22){$0$};
\node at (22,-17){$0$};
\node at (22,-12){$0$};
\node at (27,-57){$\mathbb{Z}$};
\node at (27,-52){$0$};
\node at (27,-47){$0$};
\node at (27,-42){$0$};
\node at (27,-37){$\mathbb{Z}$};
\node at (27,-32){$0$};
\node at (27,-27){$0$};
\node at (27,-22){$\mathbb{Z}_{2}$};
\node at (27,-17){$\mathbb{Z}$};
\node at (27,-12){$\mathbb{Z}_{3}$};
\node at (27,-7){$...$};
\node at (32,-57){$E^{6,0} _{2}$};
\node at (32,-52){$0$};
\node at (32,-47){$0$};
\node at (32,-42){$0$};
\node at (32,-37){$E^{6,4} _{2}$};
\node at (32,-32){$0$};
\node at (32,-27){$0$};
\node at (32,-22){$0$};
\node at (32,-17){$E^{6,8} _{2}$};
\node at (32,-12){$0$};
\node at (32,-7){$...$};
\node at (38,-57){$E^{7,0} _{2}$};
\node at (38,-52){$0$};
\node at (38,-47){$0$};
\node at (38,-42){$0$};
\node at (38,-37){$E^{7,4} _{2}$};
\node at (38,-32){$0$};
\node at (38,-27){$0$};
\node at (38,-22){$0$};
\node at (38,-17){$E^{7,8} _{2}$};
\node at (38,-12){$0$};
\node at (38,-7){$...$};
\node at (43,-57){$E^{8,0} _{2}$};
\node at (43,-52){$0$};
\node at (43,-47){$0$};
\node at (43,-42){$0$};
\node at (43,-37){$E^{8,4} _{2}$};
\node at (43,-32){$0$};
\node at (43,-27){$0$};
\node at (43,-22){$E^{8,7} _{2}$};
\node at (43,-17){$E^{8,8} _{2}$};
\node at (43,-12){$..$};
\node at (43,-7){$...$};
\node at (48,-57){$E^{9,0} _{2}$};
\node at (48,-52){$0$};
\node at (48,-47){$0$};
\node at (48,-42){$0$};
\node at (48,-37){$E^{9,4} _{2}$};
\node at (48,-32){$0$};
\node at (48,-27){$0$};
\node at (48,-22){$E^{9,7} _{2}$};
\node at (48,-17){$E^{9,8} _{2}$};
\node at (48,-12){$..$};
\node at (48,-7){$...$};
\node at (53,-57){$E^{10,0} _{2}$};
\node at (53,-52){$0$};
\node at (53,-47){$0$};
\node at (53,-42){$0$};
\node at (53,-37){$E^{10,4} _{2}$};
\node at (53,-32){$0$};
\node at (53,-27){$0$};
\node at (53,-22){$E^{10,7} _{2}$};
\draw [->] (4,-18)--(25,-35);
\draw [->] (27,-38)--(52,-56);
\draw [->] (3,-18)--(47,-56);
\draw [->] (4,-23)--(41,-56);
\draw [->] (3,-37)--(27,-56);
\end{tikzpicture}
\caption{page $2$}
\end{figure}
\end{center}
\begin{lem}
$E^{6,0} _{2} = H^{5} (K(\mathbb{Z},5),\mathbb{Z}) =0$
\end{lem}
\begin{proof}
Each differential map which come and go out $E^{6,0} _{2}$ are zero. Then $E^{6,0} _{2}$ will remain to page infinity and converg to $E^{6,0} _{\infty} =0$, therefore 
$$E^{6,0} _{2} = H^{6} (K(\mathbb{Z},5),\mathbb{Z}) =0$$
\end{proof}
\begin{lem}
$E^{7,0} _{2} = H^{7} (K(\mathbb{Z},5),\mathbb{Z})=0$
\end{lem}
\begin{proof}
Similarly, following the same reason $E^{7,0} _{2} = H^{7} (K(\mathbb{Z},5),\mathbb{Z})=0$.
\end{proof}
\begin{lem} \label{aa}
$E^{8,0} _{2} = H^{8} (K(\mathbb{Z},5),\mathbb{Z}) \simeq 0$
\end{lem}
\begin{proof}   
As we know, $E^{8,0} _{2}$ should be eliminated by the differential map, and thus $E^{8,0} _{2}\simeq 0$.
\end{proof}
\begin{lem}
$E^{9,0} _{2} \simeq 0$
\end{lem}
\begin{proof}
For this case $E^{9,0} _{2}$, it is easy to see that $E^{9,0} _{2} \simeq 0$.
\end{proof}
\begin{lem}\label{bb}
$E^{10,0} _{2} \simeq \mathbb{Z}_{3} \oplus \mathbb{Z}_{2}$
\end{lem}
\begin{proof}
From the page $2$, and following a short exact sequence 
$$0 \xrightarrow{}  \mathbb{Z}_2 \xrightarrow{\text{g}} \mathbb{Z} \xrightarrow{\text{f}} E^{10,0} _{2}  \xrightarrow{ }0$$

(Here $(f \circ g=0)$, and $g$ is a map which multiplicable by $2$), thus $E^{10,0} _{2} \simeq \mathbb{Z}_{2}$, but from the spectral sequence pages , there is a point $\mathbb{Z}_{3}$, which is should be die by differential map $d_{10}$, so
$E^{10,0} _{2} \simeq \mathbb{Z}_{3} \oplus \mathbb{Z}_{2}$
\end{proof}
\begin{lem}
$E^{11,0} _{2}= E^{12,0} _{2}=E^{13,0} _{2} = 0.$
\end{lem}
\begin{proof}
Similarly from Lemma \ref{aa} and \ref{bb}, this is easily derived. 
\end{proof}
\bibliographystyle{amsplain}

\begin{thebibliography}{n} 

\bibitem{a} A. Hatcher, Algebraic Topology, Cambridge University Press,2001.\\
\bibitem{b} A. Hatcher, Spectral sequences in Algebraic Topology (unfinished, online)\\
\bibitem{c} R. Mosher and M, Tangora, Cohomology Operations and Applications in Homotopy Theory, Harper \& Row, Publishers, 1968  
\bibitem{d} J. McCleary, A user's guide to spectral sequences, Second edition. Cambridge Studies in Advanced Mathematics, 58. Cambridge University Press, Cambridge, 2001.
\bibitem{e} J.P. May. A concise course in algebraic topology. Chicago Lectures in Mathematics. University of Chicago Press, Chicago, IL, 1999.
\end{thebibliography}

\end{flushleft}
\end{document}